\documentclass[12pt,a4paper]{article}

\usepackage{amsmath}
\usepackage{amssymb}
\usepackage{amsthm}
\usepackage{mathtools}
\usepackage{hyperref} 
\usepackage[usenames]{color}

\newcommand{\1}[1]{{\mathbf 1}{\{#1\}}}

\newcommand{\Z}{{\mathbb Z}}
\newcommand{\N}{{\mathbb N}}
\newcommand{\R}{{\mathbb R}}
\newcommand{\IP}{{\mathbb P}}
\newcommand{\IE}{{\mathbb E}}
\newcommand{\8}{{\infty}}

\newcommand{\LL}{{\mathcal L}}
\newcommand{\EE}{{\mathcal E}}

\newcommand{\NN}{{\mathcal N}}

\newcommand{\capa}{\mathop{\mathrm{cap}}\nolimits}
\newcommand{\osc}{\mathop{\mathrm{osc}}\nolimits}
\newcommand{\Vs}{{V_{\!*}}}
 
\allowdisplaybreaks

\newtheorem{theo}{Theorem}[section]
\newtheorem{lem}[theo]{Lemma}
\newtheorem{df}[theo]{Definition}
\newtheorem{prop}[theo]{Proposition}
\newtheorem{cor}[theo]{Corollary}
\newtheorem{rem}[theo]{Remark}
\newtheorem{example}[theo]{Example}

\title{Recurrence and transience\\
of open Jackson networks with balanced nodes}
\author{Serguei~Popov\thanks{Centro de Matem\'atica, University of 
Porto, Porto, Portugal. E-mail: \texttt{serguei.popov@fc.up.pt}}}
\date{\today}

\begin{document}

\maketitle

%
%

\begin{abstract}
Consider an open Jackson network with Poisson arrivals and exponential
services, in which no node is overloaded and exactly~$k_0$ nodes have load
equal to~$1$. We prove that the queue-length process is recurrent if and
only if $k_0\leq 2$. For that, we show that
the network satisfies a sector condition, so its
capacities are comparable to those of its symmetrisation; 
the latter is a
reversible network whose recurrence/transience
can be dealt with by standard means. 
We also discuss what the Lyapunov-function method gives in this setting;
in particular, we show that every balanced queue is empty at arbitrarily
large times, also in the transient case.
\\[.3cm]\textbf{Keywords:} Jackson network, recurrence, transience,
capacity, non-reversible Markov chains, Dirichlet form,
sector condition, Lyapunov function
\\[.3cm]\textbf{AMS 2020 subject classifications:}
Primary 60K25. Secondary 60J27, 60J46, 90B15.
\end{abstract}

\section{Introduction and results}
\label{s_intro}

An \emph{open Jackson network}~\cite{Jac57,Jac63,Kel79} on the node set
$V=\{1,\dots,d\}$ is specified by external Poisson arrival rates
$a_i\geq 0$, exponential service rates $\mu_i>0$, and a substochastic
routing matrix $P=(p_{ij})_{i,j\in V}$: a customer completing service
at~$i$ moves to~$j$ with probability~$p_{ij}$ and leaves the network with
probability $p_{i*}:=1-\sum_j p_{ij}$. Each node is a single-server queue.
The queue-length process $Q(t)=(Q_1(t),\dots,Q_d(t))$ is then a
continuous-time Markov chain on~$\Z_+^d=\{0,1,2,\ldots\}^d$ with bounded jump rates.
Throughout the paper we assume that $Q$ is irreducible
on~$\Z_+^d$.

The goal of this paper is to obtain a complete
\emph{recurrence/transience} classification of~$Q$. In queueing theory the
question traditionally asked is that of \emph{stability}
(i.e., positive recurrence, or ergodicity), 
since this is what makes long-run performance measures
meaningful; for Jackson networks it is answered by the product-form
theory~\cite{Jac57}, see also~\cite{FMMS93} for a Lyapunov-function
approach. The finer question of recurrence versus
transience of a non-ergodic system is, however, natural from the point of
view of Markov chains and random walks, and has been studied for several
classical queueing models; see e.g.~\cite{FG26,Pop25} for infinite-server
queues whose inter-arrival or service times may have infinite mean,
\cite{PW21} for two-dimensional queueing processes, \cite{MM03} for critical
random walks on two-dimensional complexes arising from polling systems,
and~\cite{FMM95,MPW17} for the Lyapunov-function approach to
recurrence and transience of random walks arising from queueing systems.
For random walks in the quarter plane with zero drift in the interior the
classification was obtained in~\cite{FMM92,AFM95}; its continuous
analogue, reflected Brownian motion in a wedge, is classified
in~\cite{VW85,Wil85} (see also~\cite{HR93}). For non-ergodic Jackson
networks, Goodman and Massey~\cite[Theorem~1]{GM84} proved that the queue
lengths at the nodes with load strictly less than~$1$ (see~\eqref{eq_SBT}
below) converge jointly in distribution to a product of geometric laws,
whereas the queue length at every other node tends to infinity in
probability.

Let $\Vs:=V\cup\{*\}$, the extra node~$*$ standing for the outside world.
The \emph{routing graph} is the directed graph on~$\Vs$ which has an arc
$(j,\ell)$, $j\neq\ell$, if either $j\in V$ and $p_{j\ell}>0$
(here $\ell\in\Vs$), or $j=*$ and $a_\ell>0$. It is elementary to check that
$Q$ is irreducible if and only if every node can be \emph{filled}
and \emph{drained}, in the terminology of~\cite{GM84}, that is, if and only
if for every $j\in V$ the routing graph contains a directed path from~$*$
to~$j$ and a directed path from~$j$ to~$*$; we use these two properties
without further comment. Vectors in~$\R^V$ are row vectors; inequalities
between them, as well as $x\wedge y$, are understood componentwise,
$\mathbf 1\in\R^V$ is the vector all of whose entries equal~$1$, and $e_j$
is the $j$-th unit vector (of~$\R^V$, or of~$\Z^d$). For $x\in\Z^d$ and
$A\subset V$ we write $x_A:=(x_j)_{j\in A}$; in particular,
$Q_A(t):=(Q_j(t))_{j\in A}$. For $y\in\Z_+^m$ we write $|y|:=\sum_iy_i$.
Since every node can be
drained, $P^n\to 0$; hence the spectral radius~$r$ of~$P$ is strictly less
than~$1$, and~$I-P$ is invertible, with $(I-P)^{-1}=\sum_{n\geq 0}P^n\geq 0$.

The \emph{traffic equations} are
\begin{equation}
\label{eq_traffic}
\lambda_j=a_j+\sum_{i\in V}\lambda_ip_{ij},\quad j\in V,
\qquad\text{that is,}\quad \lambda=a+\lambda P ,
\end{equation}
and they have the unique solution $\lambda=a(I-P)^{-1}=\sum_{n\geq 0}aP^n\geq 0$,
with $\lambda_j>0$ for every~$j$ because every node can be filled. The
number~$\lambda_j$ is the rate at which customers would arrive at~$j$ if
every node were able to forward \emph{all} of its own input; we call
$\widehat\rho_j:=\lambda_j/\mu_j$ the \emph{nominal load} of~$j$.

Equation~\eqref{eq_traffic} accounts for the flow through~$j$
as~$\lambda_j$, which is correct only as long as~$j$ can serve that flow:
a node receiving more than it can serve forwards customers at rate~$\mu_j$
and not~$\lambda_j$, so a saturated node acts on the rest of the network
as a source of rate~$\mu_j$, and~\eqref{eq_traffic} overestimates the flow
downstream of it. The correct bookkeeping is the (nonlinear)
\emph{throughput equation} of Goodman and Massey~\cite{GM84},
\begin{equation}
\label{eq_through}
\nu_j=a_j+\sum_{i\in V}(\nu_i\wedge\mu_i)p_{ij},\quad j\in V,
\qquad\text{that is,}\quad \nu=a+(\nu\wedge\mu)P,
\end{equation}
where $\nu_j$ is the total rate at which
customers arrive at~$j$, and $\nu_j\wedge\mu_j$ is the rate at which~$j$
forwards them, i.e., its actual long-run throughput.

\begin{prop}
\label{p_through}
Equation~\eqref{eq_through} has a unique solution $\nu\in\R^V$, and
$0\leq\nu\leq\lambda$. Moreover, $\nu\leq\mu$ if and only if
$\lambda\leq\mu$, and in that case $\nu=\lambda$.
\end{prop}

This is proved in~\cite{GM84} by a longer argument; since the following
proof takes only a few lines, we include it.

\begin{proof}
Fix $\theta\in(r,1)$ and put $u:=\sum_{n\geq 0}\theta^{-n}\,\mathbf 1P^n$,
the series being convergent because the spectral radius of $\theta^{-1}P$
is $r/\theta<1$. Then $u\geq\mathbf 1$ and $uP=\theta(u-\mathbf 1)\leq\theta u$,
so that for the weighted norm $\|x\|_u:=\max_j|x_j|/u_j$ we have
$|xP|\leq|x|P\leq\|x\|_u\,uP\leq\theta\|x\|_u\,u$, that is,
$\|xP\|_u\leq\theta\|x\|_u$. Since $y\mapsto y\wedge\mu$ is
componentwise $1$-Lipschitz, the map $F(\nu):=a+(\nu\wedge\mu)P$ is a
$\theta$-contraction for~$\|\cdot\|_u$, so that~\eqref{eq_through} has at
most one solution. Moreover, $F$ is monotone, $F(0)=a\geq 0$, and
$F(\lambda)\leq a+\lambda P=\lambda$. Hence the iterates $F^n(0)$ increase
and are bounded by $F^n(\lambda)\leq\lambda$, so that they converge to
some~$\nu$ with $0\leq\nu\leq\lambda$, which solves~\eqref{eq_through} by the
continuity of~$F$.

If $\lambda\leq\mu$, then $\lambda\wedge\mu=\lambda$, so~$\lambda$
solves~\eqref{eq_through}; by uniqueness $\nu=\lambda\leq\mu$. Conversely,
if $\nu\leq\mu$, then $\nu\wedge\mu=\nu$ and~\eqref{eq_through} becomes
$\nu=a+\nu P$, so $\nu=\lambda$ by uniqueness of the solution
of~\eqref{eq_traffic}, and $\lambda\leq\mu$.
\end{proof}

We call $\rho_j:=\nu_j/\mu_j$ the \emph{load} of~$j$, and we partition the network into \emph{stable}, \emph{balanced} and \emph{transient} nodes in the following way:
\begin{equation}
\label{eq_SBT}
S:=\{j:\rho_j<1\},\qquad B:=\{j:\rho_j=1\},\qquad T:=\{j:\rho_j>1\}.
\end{equation}
 Note that, 
by Proposition~\ref{p_through}, $T=\emptyset$ is
equivalent to $\lambda\leq\mu$, 
and then the two equations agree, so
that the sets~\eqref{eq_SBT} and the loads may be computed from the
\emph{linear} equation~\eqref{eq_traffic} alone.

\begin{example}
\label{ex_nominal}
When $T\neq\emptyset$ the two equations genuinely differ, 
and the set
$\{j:\widehat\rho_j>1\}$ of nominally overloaded nodes may be strictly
larger than~$T$. Consider the tandem queue $1\to 2$ 
(that is, $p_{12}=1$,
$p_{2*}=1$) with $a_2=0$ and $a_1>\mu_2>\mu_1$. Then
$\lambda=(a_1,a_1)$, so both nodes are nominally overloaded, whereas
$\nu=(a_1,\mu_1)$: node~$2$ receives only at rate $\mu_1<\mu_2$ and is in
fact stable, and $T=\{1\}$.
\end{example}

Now, we discuss the recurrence/transience classification,
which is the main goal of this paper.
First, if $B=T=\emptyset$, that is, if all loads are 
strictly less than~$1$, then~$Q$ is positive
recurrent, with the product-form stationary distribution
$\pi(x)=\prod_{j\in V}(1-\rho_j)\rho_j^{x_j}$~\cite{Jac57}; see
also~\cite[Ch.~2]{Kel79}.
If $T\neq\emptyset$, then~$Q$ is transient, 
and in fact ballistic: the
total number of customers grows at least linearly in time,
almost surely~\cite[Prop.~5.1 and Cor.~3.3]{Dai96}. 
Finally, if $B\neq\emptyset$, then $Q_i(t)\to\8$ in probability for every
$i\in B$~\cite[Theorem~1]{GM84}, 
so that~$Q$ is not positive recurrent. What is left
open is therefore the case
\begin{equation}
\label{eq_standing}
T=\emptyset \quad\text{and}\quad k_0:=|B|\geq 1, 
\end{equation}
in which $Q$ is either null recurrent or transient; \emph{this is the
situation we assume from now on}. By Proposition~\ref{p_through} we then
have $\nu=\lambda$, and we write~$\lambda$ for the common value of the two
solutions, so that $\rho_j=\lambda_j/\mu_j$ for every~$j$ and
$\lambda_i=\mu_i$ for $i\in B$. Our main result is the following.

\begin{theo}
\label{t_main}
Assume that $T=\emptyset$ and $k_0:=|B|\geq 1$.
If $k_0\leq 2$, then $Q$ is null recurrent; if
$k_0\geq 3$, then $Q$ is transient. Moreover, if $k_0\geq 3$, then
$|Q_B(t)|\to\8$ a.s., and there is $C<\8$ such that
\[
\IE_x\int_0^\8\1{|Q_B(t)|\leq n}\,dt\leq C\,(n+1)^{k_0}
\qquad\text{for all }x\in\Z_+^d\text{ and }n\geq 0 .
\]
\end{theo}

We do not expect the exponent~$k_0$ in the last bound to be sharp: by
analogy with the simple random walk and the Brownian motion in
dimension~$k_0\geq 3$, for which the expected total time spent in a ball of
radius~$n$ is of order~$n^2$, our conjecture is that
this bound should hold with $(n+1)^2$
in place of $(n+1)^{k_0}$.

The recurrence/transience classification of
Theorem~\ref{t_main} is what one would guess by comparison with the simple
random walk, which is recurrent in dimensions~$1$ and~$2$ and transient in
higher dimensions: the balanced queues play the role of the critical
directions, while the stable queues remain tight. This heuristic is,
however, far from being a proof. The two-dimensional simple random walk is
only barely recurrent, and the intuition that a two-dimensional walk with
zero drift (on average) has to be recurrent is simply wrong: there are
spatially non-homogeneous random walks in the plane with bounded increments
and zero drift everywhere which are transient~\cite{GMMW16}, and e.g.\ the random
walk on the randomly oriented Manhattan lattice in~$\Z^2$, whose drift
vanishes on average over the environment, is almost surely
transient~\cite{BRP26}. Closer to our setting, random walks in the quarter
plane with zero drift in the interior can be recurrent or transient,
depending on the reflection at the boundary~\cite{AFM95,FMM92}. For~$Q$,
the drift of a balanced queue vanishes only on average (when the stable
queues are in equilibrium), the drift does not vanish on the boundary of the
orthant, and~$Q$ is not reversible.

Theorem~\ref{t_main} also settles a question raised in~\cite{MMPW23},
where systems of $N+1$ particles on~$\Z$ performing nearest-neighbour
random walks with particle-dependent jump rates and exclusion interaction
are studied. The vector of the~$N$ gaps between consecutive particles is
the queue-length process of a Jackson network in which customers are routed
from a node to its two neighbours in $\{1,\dots,N\}$~\cite[Sect.~3]{MMPW23},
and it is irreducible under the standing assumption of~\cite{MMPW23} that all
particles jump to the right at positive rates. The system splits into
maximal stable sub-systems (``clouds''); by~\cite[Theorem~2.1]{MMPW23}, if
all clouds have the same asymptotic speed, then the gaps inside the clouds
are stable nodes and the gaps between consecutive clouds are balanced
nodes, so that $T=\emptyset$ and $k_0$ equals the number of clouds minus
one. For constant drifts, in which case every particle forms a cloud of its
own and all nodes are balanced, recurrence was proved for $N\leq 2$ (for
$N=2$ by means of~\cite{AFM95}) and transience was expected for $N=3$,
see~\cite[Example~2.18]{MMPW23}; for general systems whose clouds all have
the same speed, it was conjectured in~\cite[Conjecture~2.19]{MMPW23} that
$\liminf_{t\to\8}(X_{N+1}(t)-X_1(t))<\8$ a.s.\ when there are three clouds,
and $\liminf_{t\to\8}(X_{N+1}(t)-X_1(t))=\8$ a.s.\ when there are at least
four, where $X_1(t)<\dots<X_{N+1}(t)$ are the positions of the particles.
Since recurrence of the gap process is equivalent to the former property,
and transience to the latter, Theorem~\ref{t_main} confirms both
expectations.

Our main tool is the potential theory of non-reversible Markov chains
developed by Gaudilli\`ere and Landim~\cite{GL14}. A chain with an
invariant measure~$m$ has the same Dirichlet form as its
$m$-symmetrisation, whence $\capa\geq\capa_s$ for the corresponding
capacities; and if the generator satisfies a \emph{sector condition}, the
Dirichlet principle of~\cite{GL14} gives the reverse bound
$\capa\leq C_0\capa_s$. Here the product measure
$m(x)=\prod_j\rho_j^{x_j}$ is invariant and, since $\rho_i=1$ for
$i\in B$, constant in the coordinates indexed by~$B$; the symmetrised
network is therefore invariant under translations in those directions and
has summable weights in the others, and its recurrence is decided by an
elementary computation, the answer being $k_0\leq 2$.
Symmetrisation has been used for open Jackson networks before: Mao and
Xia~\cite[Sect.~3]{MX15} observed that the symmetrisation of an ergodic
open Jackson network with respect to its product-form stationary
distribution is again a Jackson network, and bounded the spectral gap of
the former through that of the latter. Their setting and objective are
different from ours: there the network is ergodic and~$m$ is a probability
measure, and the spectral gap is determined by the Dirichlet form alone;
here $m$ is not summable, and deciding recurrence requires, in addition, the
reverse comparison of capacities provided by the sector condition.

The point of the paper is that the passage back to~$Q$ is legitimate,
i.e., that a Jackson network satisfies a sector condition. Every
transition of~$Q$ moves a single customer along an arc of the routing
graph, and the remaining customers form a ``base'' configuration~$z$ which
the transition leaves untouched. In these coordinates the antisymmetric
part of the conductances factorises, and the conservation
law~\eqref{eq_traffic} says exactly that a certain flow on the
\emph{finite} graph~$\Vs$ is divergence free. Decomposing that flow into
cycles and lifting the cycles to the state space exhibits the
antisymmetric part as a superposition of cycles of length at most $d+1$,
each obtained by letting one additional customer travel once around a
cycle of the routing graph; cycles of bounded length are controlled by the
symmetric part through a local Poincar\'e inequality, which is the
mechanism by which cyclic random walks satisfy a sector condition,
cf.~\cite[Sect.~3.3]{KLO12} and~\cite[p.~82]{GL14}. 

We use the following conventions. All constants appearing below
($C$, $C_0$, $C_1$, $C_\flat$, $W$, $\Lambda$) are positive and finite and depend only on
the network parameters; the lowercase letter~$c$ is reserved for
conductances. We put $e_*:=0\in\Z^d$.

The rest of the paper is organised as follows. Section~\ref{s_cap}
collects what we need from the potential theory of non-reversible chains.
Section~\ref{s_sector} proves the sector condition for Jackson networks,
and hence that $Q$ is recurrent if and only if its symmetrisation is.
Section~\ref{s_sym} classifies the symmetrised network and proves
Theorem~\ref{t_main}. Section~\ref{s_lyap} discusses the Lyapunov-function
approach; it shows that every balanced queue is empty at arbitrarily large
times, also in the transient case, and bounds the tail of the time it needs
to empty. 

\section{Capacities of non-reversible Markov chains}
\label{s_cap}
Here, we collect some relevant definitions and facts from~\cite{GL14}.
Throughout this section $X$ is an irreducible continuous-time Markov chain
on a countable set~$E$ with bounded jump rates $q(x,y)$, $x\neq y$, total
rates $q(x):=\sum_{y\neq x}q(x,y)$, generator
$\LL f(x)=\sum_yq(x,y)(f(y)-f(x))$, and an invariant measure~$m$ (not necessarily of
finite total mass) with $0<m(x)<\8$ for every $x\in E$; that is,
\begin{equation}
\label{eq_inv}
\sum_{x\in E}m(x)q(x,y)=m(y)q(y)\qquad\text{for every }y\in E .
\end{equation}
This is the setting of~\cite[Sect.~2]{GL14}; in the notation used there,
$\mu=m$, $\lambda(x)=q(x)$, $r(x,y)=q(x,y)$ and $M(x)=m(x)q(x)$. We write
$\langle f,g\rangle_m:=\sum_xm(x)f(x)g(x)$ and
\[
\EE(f,g):=\langle-\LL f,g\rangle_m
\]
for finitely supported $f,g$; we
call $\EE(f,f)$ the \emph{Dirichlet form}, or the \emph{energy}, of~$f$.

Write $c(x,y):=m(x)q(x,y)$ for the \emph{conductances}, and let
\begin{equation}
\label{eq_cs}
c_s(x,y):=\tfrac12\big(c(x,y)+c(y,x)\big)=c_s(y,x)
\end{equation}
and
\begin{equation}
\label{eq_ca}
c_a(x,y):=\tfrac12\big(c(x,y)-c(y,x)\big)=-c_a(y,x)
\end{equation}
be their symmetric and antisymmetric parts. For a countable set~$A$, we call
an antisymmetric function $\vartheta:A\times A\to\R$ (that is,
$\vartheta(x,y)=-\vartheta(y,x)$ for all $x,y\in A$) a \emph{flow} on~$A$,
and we say that it is \emph{divergence free} if $\sum_y|\vartheta(x,y)|<\8$
and $\sum_y\vartheta(x,y)=0$ for every $x\in A$. By~\eqref{eq_inv}, the flow
$c_a$ on~$E$ is divergence free:
\begin{equation}
\label{eq_divfree}
\sum_yc_a(x,y)=\tfrac12\Big(m(x)q(x)-\sum_ym(y)q(y,x)\Big)=0,
\qquad x\in E .
\end{equation}
The \emph{symmetrisation} of~$X$ is the chain $q_s$ with rates
$q_s(x,y):=c_s(x,y)/m(x)$; it is $m$-reversible, has the same total rates
as~$X$, and also admits~$m$ as an invariant measure, all three assertions
being immediate from~\eqref{eq_inv}.

The starting point is that \emph{the energy does not see the
antisymmetric part of the dynamics}: for finitely supported~$f$,
\begin{equation}
\label{eq_energyid}
\EE(f,f)=\tfrac12\sum_{x,y}m(x)q(x,y)\big(f(y)-f(x)\big)^2=\EE_s(f,f),
\end{equation}
where $\EE_s$ is the Dirichlet form of~$q_s$.
 The first equality
 is~\cite[Sect.~2]{GL14}; 
the second equality holds because the middle
expression is unchanged when the names of $x$ and~$y$ are interchanged,
and therefore involves $c_s$ only.

Fix $x_0\in E$ and an increasing sequence of finite sets
$K_1\subset K_2\subset\cdots$ with $x_0\in K_1$ and
$\bigcup_nK_n=E$, and put $G_n:=E\setminus K_n$. Let
$h_n(x):=\IP_x[\,X\text{ hits }x_0\text{ before }G_n\,]$, so that
$h_n(x_0)=1$, $h_n=0$ on~$G_n$ and $\LL h_n=0$ elsewhere, and set
\[
\capa(x_0,G_n):=\EE(h_n,h_n),\qquad \capa(x_0):=\lim_n\capa(x_0,G_n),
\]
and analogously $\capa_s$ for~$q_s$. By~\cite[(2.6)]{GL14} this is the
capacity $\capa(\{x_0\},G_n)$ of~\cite[Definition~2.1]{GL14}, so that
\begin{equation}
\label{eq_caphit}
\capa(x_0,G_n)=m(x_0)q(x_0)\,
 \IP_{x_0}\big[X\text{ hits }G_n\text{ before returning to }x_0\big] ;
\end{equation}
letting $n\to\8$, we see that $X$ is transient if and only if
$\capa(x_0)>0$, cf.~\cite[(5.1)]{GL14}. Furthermore,
by~\cite[Lemma~2.5]{GL14},
\begin{equation}
\label{eq_capineq}
\capa(x_0,G_n)\geq\capa_s(x_0,G_n)\quad\text{for every }n,
\qquad\text{hence}\quad\capa(x_0)\geq\capa_s(x_0),
\end{equation}
so that $X$ is transient as soon as its symmetrisation is. Finally, since
$q_s$ is reversible, \cite[(2.11)]{GL14} gives the Dirichlet principle
\begin{equation}
\label{eq_dirichlet}
\capa_s(x_0,G_n)=\inf\big\{\EE_s(f,f):\ f(x_0)=1,\ f=0\text{ on }G_n\big\}.
\end{equation}

Recurrence requires a bound in the opposite direction
to~\eqref{eq_capineq}, and this is where the Dirichlet principle for
non-reversible chains of~\cite{GL14} comes in.

\begin{df}
\label{d_sector}
We say that~$\LL$ satisfies a \emph{sector condition} with constant
$C_0<\8$ if
\begin{equation}
\label{eq_sector}
\langle\LL f,g\rangle_m^2\leq C_0\,\EE(f,f)\,\EE(g,g)
\qquad\text{for all finitely supported }f,g:E\to\R .
\end{equation}
\end{df}

\begin{prop}
\label{p_capupper}
Suppose that $\LL$ satisfies a sector condition with constant~$C_0$. Then
\[
\capa(x_0,G_n)\leq C_0\,\capa_s(x_0,G_n)\ \text{ for every }n,
\qquad\text{hence}\quad \capa(x_0)\leq C_0\,\capa_s(x_0).
\]
In particular, if the symmetrised chain~$q_s$ is recurrent, then so
is~$X$.
\end{prop}

This is~\cite[Lemma~2.6]{GL14}, where the sector condition is required
for all $f,g$ in the domain of the generator in $L^2(m)$. Requiring it
only for finitely supported functions, as in~\eqref{eq_sector}, is enough:
the proof in~\cite{GL14} applies the sector condition to the two functions
appearing in the variational formula~\cite[Theorem~2.4]{GL14}, and those,
by~\cite[Remark~3.3]{GL14}, may be taken to vanish outside the finite
set~$K_n$.

\section{The sector condition for Jackson networks}
\label{s_sector}
Throughout this section we assume~\eqref{eq_standing}, so that $\nu=\lambda$
and $\rho_j=\lambda_j/\mu_j$ for every $j\in V$. We now specialise to $X=Q$
and $E=\Z_+^d$. Recalling that $e_*=0$, the transition rates of~$Q$ are,
for $x\in\Z_+^d$,
\[
q(x,x+e_\ell)=a_\ell,\qquad
q(x,x-e_j+e_\ell)=\mu_jp_{j\ell}\,\1{x_j\geq 1},
\]
for $\ell\in V$, respectively for $j\in V$ and $\ell\in\Vs\setminus\{j\}$,
all other off-diagonal rates being zero: these are the external arrivals,
and the service completions at~$j$ followed by a routing to~$\ell$ (or by a
departure from the network, if $\ell=*$). A service completion at~$j$ routed
back to~$j$ leaves the state unchanged and is therefore not a transition.
Let
\[
m(x):=\prod_{j\in V}\rho_j^{x_j},\qquad x\in\Z_+^d ;
\]
this is an invariant measure for~$Q$ by the partial balance property of
Jackson networks, cf.~\cite[Ch.~2]{Kel79} (the verification
of~\eqref{eq_inv} is purely algebraic and does not use the summability
of~$m$). Since $\rho_i=1$ for $i\in B$, we have
$m(x)=\prod_{j\in S}\rho_j^{x_j}$, a quantity depending on~$x_S$ only.
As in Section~\ref{s_cap}, the conductances are $c(x,y)=m(x)q(x,y)$.

\subsection{Transitions, and the base of a transition}

For $j,\ell\in\Vs$ define the \emph{node fluxes}
\begin{equation}
\label{eq_kappa}
\kappa(*,\ell):=a_\ell,\qquad
\kappa(j,\ell):=\lambda_jp_{j\ell}\ \ (j\in V,\ \ell\in\Vs),\qquad
\kappa(*,*):=0 ,
\end{equation}
and put
\[
\Phi(j,\ell):=\kappa(j,\ell)-\kappa(\ell,j),\qquad
\Psi(j,\ell):=\kappa(j,\ell)+\kappa(\ell,j),\qquad j,\ell\in\Vs .
\]
Thus $\kappa(j,\ell)$ is the equilibrium rate at which customers make the
move $j\to \ell$; since $\lambda_j>0$ for every $j\in V$, for $j\neq\ell$ we
have $\kappa(j,\ell)>0$ if and only if $(j,\ell)$ is an arc of the routing
graph.

\begin{lem}
\label{l_flowbal}
$\Phi$ is a divergence-free flow on~$\Vs$, that is, it is antisymmetric and
$\sum_{\ell\in\Vs}\Phi(j,\ell)=0$ for every $j\in\Vs$. Moreover
$|\Phi|\leq\Psi$.
\end{lem}

\begin{proof}
For $j\in V$ we have $\sum_\ell\kappa(j,\ell)
=\lambda_j\big(\sum_{\ell\in V}p_{j\ell}+p_{j*}\big)=\lambda_j$ and
$\sum_\ell\kappa(\ell,j)=a_j+\sum_{\ell\in V}\lambda_\ell p_{\ell j}=\lambda_j$
by~\eqref{eq_traffic}. For $j=*$ we have
$\sum_\ell\kappa(*,\ell)=\sum_\ell a_\ell$, whereas, again by~\eqref{eq_traffic},
\[
\sum_{\ell\in V}\kappa(\ell,*)=\sum_{\ell\in V}\lambda_\ell
 -\sum_{\ell,n\in V}\lambda_\ell p_{\ell n}
 =\sum_{\ell\in V}\lambda_\ell-\sum_{n\in V}(\lambda_n-a_n)=\sum_na_n .
\]
The last claim holds because $\kappa\geq 0$.
\end{proof}

Define $\mathcal D:=\{e_\ell-e_j:\ j\neq \ell\in\Vs\}\subset\Z^d$, recalling
the convention $e_*=0$.

\begin{lem}
\label{l_base}
\begin{itemize}
\item[(a)] Let $x\neq y$ in~$\Z_+^d$. If $y-x\notin\mathcal D$, then
$q(x,y)=q(y,x)=0$. If $y-x\in\mathcal D$, then there is a unique triple
$(z,j,\ell)$ with $z\in\Z_+^d$ and $j\neq\ell$ in~$\Vs$ such that
\[
x=z+e_j,\qquad y=z+e_\ell ,
\]
namely, $z=x\wedge y$. Being symmetric in $x$ and~$y$, the
configuration~$z$ depends on the unordered pair $\{x,y\}$ only; we call it
the \emph{base} of $\{x,y\}$.
\item[(b)] For every $z\in\Z_+^d$ and every $j\neq \ell$ in~$\Vs$,
\begin{equation}
\label{eq_cbase}
c(z+e_j,z+e_\ell)=m(z)\,\kappa(j,\ell),
\end{equation}
and consequently, by~\eqref{eq_ca} and~\eqref{eq_cs},
\begin{equation}
\label{eq_csca}
c_s(z+e_j,z+e_\ell)=\tfrac12m(z)\Psi(j,\ell),\qquad
c_a(z+e_j,z+e_\ell)=\tfrac12m(z)\Phi(j,\ell).
\end{equation}
Moreover, $c(x,y)=c_s(x,y)=c_a(x,y)=0$ whenever $x\neq y$ and
$y-x\notin\mathcal D$.
\end{itemize}
\end{lem}

\begin{proof}
(a) The first claim is immediate from the list of transition rates. For the
second one, note that $e_j\wedge e_\ell=0$ for $j\neq\ell$ in~$\Vs$. Hence,
if $y-x=e_\ell-e_j$, then $z:=x\wedge y$ satisfies $x-z=(x-y)^+=e_j$ and
$y-z=(y-x)^+=e_\ell$; and any representation $x=z'+e_{j'}$, $y=z'+e_{\ell'}$
with $j'\neq\ell'$ forces $z'=x\wedge y=z$, and then $e_{j'}=e_j$,
$e_{\ell'}=e_\ell$.

(b) Write $\rho_*:=1$, so that $m(z+e_j)=\rho_jm(z)$ for every $j\in\Vs$.
If $j=*$, then $q(z,z+e_\ell)=a_\ell$, so that
$c(z,z+e_\ell)=m(z)a_\ell=m(z)\kappa(*,\ell)$. If $j\in V$, then
$q(z+e_j,z+e_\ell)=\mu_jp_{j\ell}$, because $(z+e_j)_j\geq 1$; hence
$c(z+e_j,z+e_\ell)=m(z)\rho_j\mu_jp_{j\ell}=m(z)\lambda_jp_{j\ell}$, since
$\rho_j\mu_j=\lambda_j$. This is~\eqref{eq_cbase}, and~\eqref{eq_csca}
follows. The last claim follows from part~(a).
\end{proof}

Thus the transitions of~$Q$ with base~$z$ are exactly the moves of
\emph{one} customer, added to the configuration~$z$, along an arc of the
routing graph, and their conductances are $m(z)$ times the node
fluxes. For $z\in\Z_+^d$ we write
\[
\NN(z):=\{z+e_n:\ n\in\Vs\}\subset\Z_+^d .
\]

\subsection{The antisymmetric part is a superposition of short cycles}

Let $A$ be a countable set. A \emph{simple cycle} in~$A$ is a sequence
$\gamma=(v_0,v_1,\dots,v_L=v_0)$ of elements of~$A$ with $L\geq 3$ and
$v_0,\dots,v_{L-1}$ pairwise distinct; we write $L_\gamma:=L$ for its
length. Its \emph{cycle flow} is the flow on~$A$ given by
\[
\chi_\gamma(x,y):=\sum_{i<L}\big(\1{(x,y)=(v_i,v_{i+1})}
 -\1{(x,y)=(v_{i+1},v_i)}\big),\qquad x,y\in A,
\]
cf.~\cite[(2.14)]{GL14}. That is, $\chi_\gamma(x,y)=1$ if $(x,y)$ is one of
the steps $(v_i,v_{i+1})$ of~$\gamma$, $\chi_\gamma(x,y)=-1$ if $(y,x)$ is,
and $\chi_\gamma(x,y)=0$ otherwise (the first two cases exclude each other
because $L\geq 3$); clearly, $\chi_\gamma$ is divergence free. We shall use
simple cycles in $A=\Vs$, whose lengths are at most $d+1$, and in
$A=\Z_+^d$.

\begin{lem}
\label{l_flowdec}
There are finitely many simple cycles $\gamma_1,\dots,\gamma_M$ in~$\Vs$
and weights $w_1,\dots,w_M>0$ such that
$\Phi=\sum_{u\leq M}w_u\chi_{\gamma_u}$, and such that $\Phi$ is strictly
positive along every step of every~$\gamma_u$. In particular, since
$\kappa\geq\Phi$, every step of every~$\gamma_u$ is an arc of the routing
graph.
\end{lem}

\begin{proof}
This is the standard decomposition of a divergence-free flow on a finite
set; we recall the argument. We induct on the number of pairs $(j,\ell)$
with $\Phi(j,\ell)\neq 0$. If $\Phi\neq 0$, pick $(n_0,n_1)$ with
$\Phi(n_0,n_1)>0$. Since $\sum_\ell\Phi(n_1,\ell)=0$ and $\Phi(n_1,n_0)<0$,
there is~$n_2$ with $\Phi(n_1,n_2)>0$; continuing in this way, the
finiteness of~$\Vs$ forces a repetition, and the first repetition closes a
simple cycle~$\gamma$ along all of whose steps $\Phi$ is strictly positive.
Its length is at least~$3$, because $\Phi(j,j)=0$, and $\Phi(j,\ell)>0$ and
$\Phi(\ell,j)>0$ cannot hold simultaneously. Setting $w:=\min\Phi$ over the
steps of~$\gamma$, the flow $\Phi-w\chi_\gamma$ is divergence free, it is
strictly positive only on pairs where~$\Phi$ is, and it vanishes on at
least one more pair.
\end{proof}

For a simple cycle $\gamma=(n_0,\dots,n_L)$ in~$\Vs$ and $z\in\Z_+^d$, the
\emph{lift} $\gamma^z:=(z+e_{n_0},\dots,z+e_{n_L})$ is a simple cycle
in~$\Z_+^d$ with vertices in~$\NN(z)$.
Then,
$\chi_{\gamma^z}(z+e_j,z+e_\ell)=\chi_\gamma(j,\ell)$ for all $j\neq\ell$
in~$\Vs$; and, every step of~$\gamma^z$ being of the form
$(z+e_j,z+e_\ell)$, Lemma~\ref{l_base}(a) shows that
$\chi_{\gamma^z}(x,y)=0$ unless $y-x\in\mathcal D$ and $x\wedge y=z$.
By~\eqref{eq_cbase}, if every step of~$\gamma$ is an arc of the routing
graph (as is the case for the cycles~$\gamma_u$ of Lemma~\ref{l_flowdec}),
then every step of~$\gamma^z$ is a transition of~$Q$. It is worth stressing
that the lift of a cycle passing through~$*$ is again a \emph{closed} cycle
of~$\Z_+^d$: the extra customer enters the network at~$z$ and later leaves
it, returning the configuration to~$z$.

\begin{lem}
\label{l_cycles}
With $\gamma_u,w_u$ as in Lemma~\ref{l_flowdec},
\begin{equation}
\label{eq_cadec}
c_a(x,y)=\tfrac12\,m(x\wedge y)\sum_{u\leq M}w_u\,
 \chi_{\gamma_u^{x\wedge y}}(x,y)\qquad\text{for all }x\neq y\in\Z_+^d .
\end{equation}
\end{lem}

\begin{proof}
If $y-x\notin\mathcal D$, both sides vanish, by Lemma~\ref{l_base}(b) and
the preceding paragraph. Otherwise, $x=z+e_j$ and $y=z+e_\ell$ with
$z=x\wedge y$ and $j\neq\ell$, and the right-hand side
of~\eqref{eq_cadec} equals
$\frac12m(z)\sum_uw_u\chi_{\gamma_u}(j,\ell)=\frac12m(z)\Phi(j,\ell)$, which
is $c_a(x,y)$ by~\eqref{eq_csca}.
\end{proof}

\subsection{A local Poincar\'e inequality, and the sector condition}

Put
\[
C_\flat:=\tfrac12\min\big\{\kappa(j,\ell):\ (j,\ell)\text{ is an arc of the
routing graph}\big\}>0 .
\]
By~\eqref{eq_csca}, and since $\Psi\geq\kappa$, every pair
$\{z+e_j,z+e_\ell\}$ such that $(j,\ell)$ is an arc of the routing graph
satisfies $c_s(z+e_j,z+e_\ell)\geq C_\flat\,m(z)$.

By an \emph{edge} we mean an unordered pair $b=\{x,y\}$ of distinct points
of~$\Z_+^d$, and we write $c_s(b):=c_s(x,y)$. For $f:\Z_+^d\to\R$ we write
$\osc_Af:=\max_Af-\min_Af$ and $(\nabla_bf)^2:=(f(y)-f(x))^2$.
By~\eqref{eq_energyid}, for finitely supported~$f$,
\begin{equation}
\label{eq_Dedges}
\EE(f,f)=\sum_bc_s(b)(\nabla_bf)^2
 =\sum_{z\in\Z_+^d}\ \sum_{b\text{ with base }z}c_s(b)(\nabla_bf)^2,
\end{equation}
the sums being over edges; the second equality holds because, by
Lemma~\ref{l_base}, $c_s$ vanishes on the edges which have no base, and
every other edge has exactly one base.

\begin{lem}
\label{l_poincare}
For every finitely supported~$f$,
\[
\sum_{z\in\Z_+^d}m(z)\,\big(\osc_{\NN(z)}f\big)^2
 \leq\frac{4d}{C_\flat}\,\EE(f,f).
\]
\end{lem}

\begin{proof}
Fix~$z$ and $j\in V$. Since $j$ can be drained, the routing graph contains
a directed path $j=n_0,n_1,\dots,n_r=*$ with $n_0,\dots,n_r$ pairwise
distinct, so that $r\leq d$. Its lift $z+e_{n_0},\dots,z+e_{n_r}=z$ consists
of the edges $b_i:=\{z+e_{n_i},z+e_{n_{i+1}}\}$, $i<r$, which have
base~$z$ and satisfy $c_s(b_i)\geq C_\flat m(z)$. Hence, by the
Cauchy--Schwarz inequality,
\[
m(z)\big(f(z+e_j)-f(z)\big)^2\leq r\,m(z)\sum_{i<r}(\nabla_{b_i}f)^2
\leq\frac{d}{C_\flat}\sum_{b\text{ with base }z}c_s(b)(\nabla_bf)^2 .
\]
Since $\osc_{\NN(z)}f\leq 2\max_{j\in V}|f(z+e_j)-f(z)|$, the same bound
with an extra factor~$4$ holds for $m(z)(\osc_{\NN(z)}f)^2$. It remains to
sum over~$z$ and to use~\eqref{eq_Dedges}.
\end{proof}

\begin{prop}
\label{p_sector}
The generator of~$Q$ satisfies the sector condition~\eqref{eq_sector} with
respect to~$m$, with
\[
C_0=\Big(1+\frac{4d\,W}{C_\flat}\Big)^2,
\qquad W:=\sum_{u\leq M}w_uL_{\gamma_u} ,
\]
where $M$, $\gamma_u$ and $w_u$ are as in Lemma~\ref{l_flowdec}.
\end{prop}

\begin{proof}
Let $f,g$ be finitely supported, so that all the sums below are finite.
Since $m(x)\LL f(x)=\sum_yc(x,y)(f(y)-f(x))$ and $c=c_s+c_a$,
\begin{align*}
\langle\LL f,g\rangle_m&=\sum_{x,y}c_s(x,y)\big(f(y)-f(x)\big)g(x)\\
&\qquad+\sum_{x,y}c_a(x,y)f(y)g(x)
 -\sum_xf(x)g(x)\sum_yc_a(x,y).
\end{align*}
The last sum vanishes by~\eqref{eq_divfree}. By the symmetry of~$c_s$,
the first one equals
\[
-\tfrac12\sum_{x,y}c_s(x,y)\big(f(y)-f(x)\big)\big(g(y)-g(x)\big),
\]
whose absolute value is at most 
$\big(\EE(f,f)\EE(g,g)\big)^{1/2}$ by the
Cauchy--Schwarz inequality and~\eqref{eq_energyid}.

For the middle sum, we define, for a simple cycle
$\gamma=(v_0,\dots,v_L)$ in~$\Z_+^d$,
\[
\beta_\gamma(g,f):=\sum_{x,y}\chi_\gamma(x,y)\,g(x)f(y)
=\sum_{i<L}\big[g(v_i)f(v_{i+1})-g(v_{i+1})f(v_i)\big].
\]
Since $\chi_{\gamma^z}(x,y)=0$ unless $z=x\wedge y$, Lemma~\ref{l_cycles}
gives
\[
\sum_{x,y}c_a(x,y)f(y)g(x)
 =\tfrac12\sum_zm(z)\sum_{u\leq M}w_u\,\beta_{\gamma_u^z}(g,f),
\]
where only finitely many~$z$ contribute. The bilinear form $\beta_\gamma$
vanishes as soon as one of its two arguments is constant on the set
$\{v_0,\dots,v_{L-1}\}$: if $f\equiv\bar f$ there, then
$\beta_\gamma(g,f)=\bar f\sum_i(g(v_i)-g(v_{i+1}))=0$, and symmetrically.
Hence, subtracting from~$f$ and~$g$ their averages over the vertices of
$\gamma_u^z$, which lie in~$\NN(z)$, and bounding each of the $2L_{\gamma_u}$
resulting products,
\[
\big|\beta_{\gamma_u^z}(g,f)\big|
 \leq 2L_{\gamma_u}\,\osc_{\NN(z)}g\ \osc_{\NN(z)}f .
\]
Therefore, by the Cauchy--Schwarz inequality and
Lemma~\ref{l_poincare},
\begin{align*}
\Big|\sum_{x,y}c_a(x,y)f(y)g(x)\Big|
&\leq W\sum_zm(z)\,\osc_{\NN(z)}f\ \osc_{\NN(z)}g\\
&\leq\frac{4dW}{C_\flat}\,\EE(f,f)^{1/2}\EE(g,g)^{1/2}.
\end{align*}
Adding the two bounds gives
$|\langle\LL f,g\rangle_m|\leq\sqrt{C_0}\,\EE(f,f)^{1/2}\EE(g,g)^{1/2}$.
\end{proof}

\begin{theo}
\label{t_compare}
Assume~\eqref{eq_standing}, let $q_s$ be the symmetrisation of~$Q$ with respect to~$m$,
let $G_n$ be a sequence of sets as in Section~\ref{s_cap},
and let $C_0$ be the constant of
Proposition~\ref{p_sector}. Then
\[
\capa_s(x_0,G_n)\ \leq\ \capa(x_0,G_n)\ \leq\ C_0\,\capa_s(x_0,G_n)
\quad\text{for every }n .
\]
In particular, $Q$ is recurrent if and only if its symmetrisation is
recurrent.
\end{theo}

\begin{proof}
Combine~\eqref{eq_capineq} with Propositions~\ref{p_sector}
and~\ref{p_capupper}.
\end{proof}

\section{The symmetrised network, and the proof of Theorem~\ref{t_main}}
\label{s_sym}

By~\eqref{eq_cs}, \eqref{eq_csca} and Lemma~\ref{l_base}(b), the
symmetrisation of~$Q$ is the reversible chain on~$\Z_+^d$ with conductances
\begin{equation}
\label{eq_symcond}
c_s(z+e_j,z+e_\ell)=\tfrac12\,m(z)\,\Psi(j,\ell),\quad
z\in\Z_+^d,\ j\neq \ell\in\Vs,
\end{equation}
all other conductances being zero. 
Note that the process with rates~$q_s$ is itself the queue-length process of an open Jackson
network, with the same service rates~$\mu_j$, external arrival rates
$a^s_\ell:=\frac12\Psi(*,\ell)$, $\ell\in V$, and routing probabilities
\[
p^s_{j\ell}:=\frac{\Psi(j,\ell)}{2\lambda_j},\qquad j\in V,\ \ell\in\Vs
\]
(so that $p^s_{jj}=p_{jj}$). Indeed, $m(z+e_j)=\rho_jm(z)$ for $j\in V$,
so that~\eqref{eq_symcond} gives $q_s(z+e_j,z+e_\ell)=\mu_jp^s_{j\ell}$
and $q_s(z,z+e_\ell)=a^s_\ell$. By the proof of Lemma~\ref{l_flowbal},
$\sum_{\ell\in\Vs}\Psi(j,\ell)=2\lambda_j$ for $j\in V$; hence
$\sum_{\ell\in\Vs}p^s_{j\ell}=1$, and, $\Psi$ being symmetric, $\lambda$
solves the traffic equations $\lambda=a^s+\lambda P^s$ of this network.
The latter therefore has the same loads~$\rho_j$, and the same sets~$B$
and~$S$, as the original one; it is reversible, because
$\lambda_jp^s_{j\ell}=\frac12\Psi(j,\ell)$ is symmetric in~$j$ and~$\ell$.
(For ergodic networks this was observed in~\cite[Sect.~3]{MX15}.) Thus,
Theorem~\ref{t_compare} reduces the question of recurrence to reversible
Jackson networks with the same~$B$ and~$S$.
The essential point is that this
network is invariant under the translations $x\mapsto x+e_i$, $i\in B$,
because~$m(z)$ does not depend on~$z_B$, whereas its weights are summable in
the coordinates indexed by~$S$, because $\rho_j<1$ for $j\in S$.

Call $x,y\in\Z_+^d$ \emph{neighbours} if $c_s(x,y)>0$; every point has at
most $d(d+1)$ neighbours, and the resulting graph is connected. The jump
chain of~$q_s$ is the random walk on the network $(\Z_+^d,c_s)$ in the sense
of~\cite[Sect.~2.1]{LP16}. The probability in~\eqref{eq_caphit} depends on
the jump chain only, and $m(x_0)q_s(x_0)=\sum_yc_s(x_0,y)$. Hence,
by~\eqref{eq_caphit} applied to~$q_s$ and by~\cite[(2.4)]{LP16},
$\capa_s(x_0,G_n)$ is the effective conductance between~$x_0$ and~$G_n$ in
this network, and $\capa_s(x_0)$ is the effective conductance from~$x_0$ to
infinity, which does not depend on the choice of the sets~$K_n$;
see~\cite[Sect.~2.2 and Exercise~2.4]{LP16}. We shall use the following two
classical criteria~\cite[Sect.~2.5]{LP16}; recall that an edge~$b$ is an
unordered pair of distinct points, and $c_s(b)$ its conductance.

\emph{The Nash-Williams criterion}~\cite[(2.14)]{LP16}: if
$\Pi_1,\Pi_2,\ldots$ are pairwise disjoint finite sets of edges with
positive conductances, each of which separates~$x_0$ from infinity (that
is, every infinite path of neighbours which starts at~$x_0$ and visits every
point at most once uses an edge of each~$\Pi_i$), then
\[
\frac{1}{\capa_s(x_0)}\geq\sum_i\frac{1}{c_s(\Pi_i)},
\qquad\text{where}\quad c_s(\Pi_i):=\sum_{b\in\Pi_i}c_s(b)
\]
(and $1/0:=\8$).

\emph{Lyons' criterion}~\cite[Theorems~2.3 and~2.11]{LP16}:
$\capa_s(x_0)>0$ if and only if there is a \emph{unit flow from~$x_0$ to
infinity of finite energy}, that is, a flow~$\vartheta$ on~$\Z_+^d$ which
vanishes on pairs of non-neighbours, satisfies
$\sum_y\vartheta(x,y)=\1{x=x_0}$ for every~$x$, and has
\[
\|\vartheta\|^2:=\sum_{b=\{x,y\}:\,c_s(b)>0}\frac{\vartheta(x,y)^2}{c_s(b)}<\8 .
\]

\begin{proof}[Proof of Theorem~\ref{t_main}.]
By Theorem~\ref{t_compare}, $Q$ is recurrent if and only if its
symmetrisation is; moreover, $Q$ is not positive recurrent, because
$B\neq\emptyset$ (see the discussion preceding~\eqref{eq_standing}). So it
is enough to show that the symmetrisation of~$Q$ is recurrent if
$k_0\leq 2$, and transient if $k_0\geq 3$. We take $x_0:=0$.

Let $k_0\leq 2$. For $i\geq 0$, let $\Pi_i$ be the set of edges with
positive conductance joining $\{x:|x|=i\}$ to $\{x:|x|=i+1\}$. Since~$|x|$
changes by at most~$1$ between neighbours, by Lemma~\ref{l_base}(a), and
the sets $\{x:|x|=i\}$ are finite, the sets~$\Pi_i$ are finite, pairwise
disjoint, and each of them separates~$0$ from infinity. Every edge
of~$\Pi_i$ has an endpoint~$x$ with $|x|=i$, and
$\sum_yc_s(x,y)=m(x)q_s(x)=m(x)q(x)\leq\Lambda m(x)$ with
$\Lambda:=\sum_ja_j+\sum_j\mu_j$. Hence, splitting a configuration~$x$
with $|x|=i$ into its $B$- and $S$-parts and using that~$m$ does not
depend on~$x_B$,
\begin{align*}
c_s(\Pi_i)&\leq\Lambda\sum_{|x|=i}m(x)
 =\Lambda\sum_{r=0}^{i}\#\big\{x_B\in\Z_+^{k_0}:|x_B|=i-r\big\}
 \sum_{|x_S|=r}\prod_{j\in S}\rho_j^{x_j}\\
&\leq\Lambda\,(i+1)^{k_0-1}\prod_{j\in S}\frac{1}{1-\rho_j},
\end{align*}
because $\#\{x_B\in\Z_+^{k_0}:|x_B|=n\}=\binom{n+k_0-1}{k_0-1}$ equals~$1$
for $k_0=1$ and $n+1\leq i+1$ for $k_0=2$ (here $n=i-r$). For $k_0\leq 2$ the
series $\sum_i(i+1)^{1-k_0}$ diverges, so the Nash-Williams criterion gives
$\capa_s(0)=0$, that is, the symmetrisation of~$Q$ is recurrent.

\medskip

Let now $k_0\geq 3$. We use the classical fact that, for $k\geq 3$, the
lattice~$\Z^k$ with unit conductances carries a unit flow from the origin
to infinity of finite energy; this follows from P\'olya's theorem and
Lyons' criterion, cf.~\cite[Sect.~2.5]{LP16}. We first transfer such a flow
to the orthant.

\begin{lem}
\label{l_orthflow}
For $k\geq 3$, the graph~$\Z_+^k$ with unit conductances carries a unit
flow from the origin to infinity of finite energy.
\end{lem}

\begin{proof}
Let $\psi:\Z^k\to\Z_+^k$ be the folding map $\psi(x)_i:=|x_i|$. If
$x$ and $x+e_i$ are neighbours in~$\Z^k$, then $\psi(x)$ and $\psi(x+e_i)$
differ exactly by~$\pm e_i$, because $|x_i|$ and $|x_i+1|$ are the
absolute values of two consecutive integers and therefore differ by~$1$.
Thus $\psi$ maps edges to edges, and the image of a unit flow
$\vartheta_\Z$ from the origin to infinity,
\[
\vartheta(y,y'):=\sum_{\psi(x)=y,\ \psi(x')=y'}\vartheta_\Z(x,x'),
\]
satisfies, for every $y\in\Z_+^k$,
$\sum_{y'}\vartheta(y,y')=\sum_{x\in\psi^{-1}(y)}
\sum_{x'}\vartheta_\Z(x,x')=\1{y=0}$, since $\psi^{-1}(0)=\{0\}$. So
$\vartheta$ is a unit flow from the origin to infinity on~$\Z_+^k$.
Finally each edge of~$\Z_+^k$ has at most~$2^k$ preimages, so that
$\vartheta(y,y')^2\leq 2^k\sum\vartheta_\Z(x,x')^2$ by the
Cauchy--Schwarz inequality, and the energy of~$\vartheta$ is at
most~$2^k$ times that of~$\vartheta_\Z$.
\end{proof}

Index the coordinates of~$\Z_+^{k_0}$ by~$B$, and identify $\Z_+^{k_0}$ with
$\{x\in\Z_+^d:x_S=0\}$ by $y\mapsto\bar y$, where $\bar y_B=y$ and
$\bar y_S=0$; note that $m(\bar y)=1$. We transfer the flow~$\vartheta$ of
Lemma~\ref{l_orthflow} (with $k=k_0$) to the symmetrised network by
replacing each edge of~$\Z_+^{k_0}$ by a path.

Fix $i\in B$ and $y\in\Z_+^{k_0}$. Since $i$ can be filled, the routing
graph contains a directed path $*,n_0,n_1,\dots,n_r=i$ with
$n_0,\dots,n_r$ pairwise distinct, so that $r\leq d-1$. Consider the path
\begin{equation}
\label{eq_path}
\bar y\ \to\ \bar y+e_{n_0}\ \to\ \bar y+e_{n_1}\ \to\ \cdots\ \to\
 \bar y+e_{n_r}=\bar y+e_i ,
\end{equation}
whose successive states are neighbours in the symmetrised network: an
external arrival at~$n_0$, and then service completions at
$n_0,\dots,n_{r-1}$, each legitimate because the customer just placed is
there. All the edges of~\eqref{eq_path} have base~$\bar y$, by
Lemma~\ref{l_base}(a), and correspond to arcs of the routing graph; they
therefore carry conductance at least~$C_\flat m(\bar y)=C_\flat$. Moreover,
the path visits pairwise distinct states.

Let $\vartheta_H$ be the flow obtained from~$\vartheta$ by sending, for
every $y$ and every $i\in B$, the amount $\vartheta(y,y+e_i)$ along the
path~\eqref{eq_path}. Each path contributes a net divergence at its two
endpoints only, so $\vartheta_H$ is a unit flow from~$\bar 0=0$ to
infinity in the symmetrised network. As for its energy, every edge
with positive conductance has a unique base~$z$, by
Lemma~\ref{l_base}(a), and it can occur in the path~\eqref{eq_path} only
for $y=z_B$ and only if $z_S=0$; so it is used by at most~$k_0$ of the
paths. Since moreover at most $(d+1)^2$ edges share a given base,
\[
\|\vartheta_H\|^2\leq\frac{1}{C_\flat}\sum_b\vartheta_H(b)^2
\leq\frac{(d+1)^2k_0}{C_\flat}\sum_{y\in\Z_+^{k_0}}\sum_{i\in B}
 \vartheta(y,y+e_i)^2<\8
\]
by the Cauchy--Schwarz inequality and Lemma~\ref{l_orthflow}, where
$\vartheta_H(b)^2:=\vartheta_H(x,x')^2$ for $b=\{x,x'\}$. By Lyons'
criterion, $\capa_s(0)>0$, that is, the symmetrisation of~$Q$ is transient.

\medskip

It remains to prove the last assertion of Theorem~\ref{t_main}; let
$k_0\geq 3$ and $x\in\Z_+^d$. Let~$Q^*$ be the time reversal of~$Q$ with
respect to~$m$, that is, the chain with rates
$q^*(y,y'):=m(y')q(y',y)/m(y)$. By~\eqref{eq_inv}, $q^*(y)=q(y)$ for
every~$y$; hence~$Q^*$ is irreducible, has bounded jump rates and
admits~$m$ as an invariant measure, so that it fits the setting of
Section~\ref{s_cap}. Its conductances are $c^*(y,y')=c(y',y)$, so that its
$m$-symmetrisation is again~$q_s$. Since~$q_s$ is transient, $\capa_s(x)>0$
by~\eqref{eq_caphit}, and~\eqref{eq_capineq} applied to~$Q^*$ gives
$\capa^*(x)\geq\capa_s(x)$, where $\capa^*$ is the capacity for~$Q^*$.

Let $G(y,y'):=\IE_y\int_0^\8\1{Q(t)=y'}\,dt$, and let~$G^*$ be the
analogous quantity for~$Q^*$. Since
$m(y)\IP_y[Q^*(t)=y']=m(y')\IP_{y'}[Q(t)=y]$ for all $t\geq 0$, we have
$m(x)G(x,y)=m(y)G^*(y,x)$. By the strong Markov property at the hitting
time of~$x$, and by~\eqref{eq_caphit} for~$Q^*$ (letting $n\to\8$ there),
\[
G^*(y,x)\leq G^*(x,x)
 =\frac{1}{q(x)\,\IP^*_x\big[Q^*\text{ never returns to }x\big]}
 =\frac{m(x)}{\capa^*(x)}\leq\frac{m(x)}{\capa_s(x)} .
\]
Therefore, with $A_n:=\{y\in\Z_+^d:|y_B|\leq n\}$,
\[
\IE_x\int_0^\8\1{|Q_B(t)|\leq n}\,dt=\sum_{y\in A_n}G(x,y)
 =\sum_{y\in A_n}\frac{m(y)}{m(x)}\,G^*(y,x)\leq\frac{m(A_n)}{\capa_s(x)},
\]
and, as~$m$ does not depend on~$y_B$,
\[
m(A_n)=\#\big\{y_B\in\Z_+^{k_0}:|y_B|\leq n\big\}
 \prod_{j\in S}\frac{1}{1-\rho_j}
 \leq(n+1)^{k_0}\prod_{j\in S}\frac{1}{1-\rho_j},
\]
because $\binom{n+k_0}{k_0}=\prod_{u=1}^{k_0}\frac{n+u}{u}\leq(n+1)^{k_0}$.
To make this bound uniform in~$x$, note that the network started
at~$0$ is componentwise dominated by the network started at~$x$,
so that the former spends at
least as much time in~$A_n$ as the latter. Therefore, by the above bound
with $x=0$,
\[
\IE_x\int_0^\8\1{|Q_B(t)|\leq n}\,dt\leq
\IE_0\int_0^\8\1{|Q_B(t)|\leq n}\,dt\leq\frac{m(A_n)}{\capa_s(0)}
\leq C\,(n+1)^{k_0},
\]
with $C:=\capa_s(0)^{-1}\prod_{j\in S}(1-\rho_j)^{-1}$. Finally, conditionally
on the jump chain of~$Q$, the holding times of~$Q$ at the successive states
of~$A_n$ that it visits are independent exponential random variables with
rates at most~$\Lambda$; if the jump chain visited~$A_n$ infinitely often,
the total time spent by~$Q$ in~$A_n$ would therefore be infinite a.s., which
is excluded by the above bound. Hence $|Q_B(t)|>n$ for all~$t$ large
enough, a.s.; since~$n$ is arbitrary, $|Q_B(t)|\to\8$ a.s.

This concludes the proof of Theorem~\ref{t_main}.
\end{proof}

\section{Lyapunov functions and the balanced queues}
\label{s_lyap}

Theorem~\ref{t_main} settles the question of recurrence and transience
of~$Q$, but its proof is not constructive: it compares capacities, it does
not produce any explicit function of the state which would be a super- or
submartingale for~$Q$, and it says little about the individual balanced
queues. The Lyapunov-function method~\cite{FMM95,MPW17} is based on such
explicit functions, which usually yield quantitative information on the
trajectories as well; for instance, an explicit Lyapunov function for
transience would give bounds on the probability of ever returning, from far
away, to a neighbourhood of the origin, which decay polynomially in the
initial distance. In this section we explain how this method applies in
the situation~\eqref{eq_standing}. Its main ingredient is a
\emph{corrector}: a function of the stable queues which, added to the
``naive'' function of the balanced queues, removes the fluctuating drift
produced by the stable part of the network. As a consequence, we show that
\emph{every} balanced queue becomes empty at arbitrarily large times, for
every~$k_0$ (so also in the transient case), and we bound the tail of the
time it needs to empty (Proposition~\ref{p_tail} and
Corollary~\ref{c_turns}). The simultaneous emptying of all balanced queues
is a genuinely $k_0$-dimensional question, which we only discuss briefly
(Remarks~\ref{r_k2} and~\ref{r_k3}). In this section~$\LL$ is the generator
of~$Q$, applied componentwise to vector- and matrix-valued functions,
and~$^\top$ denotes transposition.

\medskip

Consider the \emph{routing chain}, that is, the Markov chain on~$\Vs$ which
jumps from $j\in V$ to $\ell\in\Vs$ with probability~$p_{j\ell}$ and is
absorbed at~$*$; since $P^n\to 0$, it is absorbed a.s. For $n\in\Vs$ and
$i\in B$ let $\Gamma_{ni}$ be the probability that the routing chain started
at~$n$ visits~$B$ (at some time $t\geq 0$) and that its first visit to~$B$
is at~$i$; thus $\Gamma_{ni}=\1{n=i}$ for $n\in B$, and $\Gamma_{*i}=0$.
Write $\Gamma_n:=(\Gamma_{ni})_{i\in B}\in[0,1]^B$, and, for $i,j\in B$,
let~$\tilde p_{ij}$ be the probability that the routing chain started at~$i$
returns to~$B$ and that its first return is to~$j$; the matrix
$\tilde P:=(\tilde p_{ij})_{i,j\in B}$ is the transition matrix of the
routing chain watched on~$B$, and $\tilde p_{i\cdot}:=(\tilde p_{ij})_{j\in B}$
denotes its $i$-th row. By the Markov property at the first step,
\begin{equation}
\label{eq_harm}
\sum_{\ell\in\Vs}p_{n\ell}\,\Gamma_\ell=\Gamma_n\quad(n\in S),
\qquad
\sum_{\ell\in\Vs}p_{i\ell}\,\Gamma_\ell=\tilde p_{i\cdot}\quad(i\in B).
\end{equation}
Moreover, $\tilde p_{ii}<1$ for $i\in B$, since otherwise the routing chain
started at~$i$ would return to~$i$ infinitely often a.s. Let us define
the vector-valued function $Y:\Z_+^d\to\R_+^B$ by
\[
Y(x):=\sum_{n\in V}x_n\Gamma_n=x_B+u(x_S),\qquad
u(x_S):=\sum_{n\in S}x_n\Gamma_n,\qquad x\in\Z_+^d ,
\]
where $x_B=(x_i)_{i\in B}\in\Z_+^B$; that is,
\[
Y_i(x)=x_i+\sum_{n\in S}x_n\Gamma_{ni},\qquad i\in B .
\]
Thus $Y_i(x)$ counts the customers at~$i$, plus the customers at the
stable nodes, each weighted by the probability that it will enter~$B$ at~$i$
next. The function~$Y$ has linear growth, and each coordinate of each of its
jumps lies in $[-1,1]$. Finally, recalling that $\Gamma_i=(\1{j=i})_{j\in B}$
for $i\in B$, let
\[
r_i:=\mu_i\big(\Gamma_i-\tilde p_{i\cdot}\big)\in\R^B,\qquad i\in B .
\]

\begin{lem}
\label{l_drift}
For every $x\in\Z_+^d$,
\[
\LL Y(x)=\sum_{i\in B}\1{x_i=0}\,r_i .
\]
In particular, $\LL Y(x)=0$ whenever $x_i\geq 1$ for all $i\in B$.
\end{lem}

\begin{proof}
An external arrival at~$\ell$ moves~$Y$ by~$\Gamma_\ell$, and a service
completion at~$j$ routed to $\ell\in\Vs$ moves it by $\Gamma_\ell-\Gamma_j$.
Hence, by~\eqref{eq_harm},
\[
\LL Y(x)=\sum_{\ell\in V}a_\ell\Gamma_\ell
 +\sum_{j\in V}\mu_j\1{x_j\geq 1}\Big(\sum_{\ell\in\Vs}p_{j\ell}\Gamma_\ell-\Gamma_j\Big)
 =\sum_{\ell\in V}a_\ell\Gamma_\ell-\sum_{i\in B}\1{x_i\geq 1}\,r_i ,
\]
the terms with $j\in S$ being equal to zero. It remains to show that
$\sum_\ell a_\ell\Gamma_\ell=\sum_{i\in B}r_i$. Multiply~\eqref{eq_traffic}
by~$\Gamma_\ell$ and sum over~$\ell$; since $\Gamma_*=0$, \eqref{eq_harm}
gives
\[
\sum_{n\in S}\lambda_n\Gamma_n+\sum_{i\in B}\lambda_i\Gamma_i
=\sum_{\ell\in V}a_\ell\Gamma_\ell+\sum_{n\in S}\lambda_n\Gamma_n
 +\sum_{i\in B}\lambda_i\tilde p_{i\cdot} ,
\]
and we conclude by recalling that $\lambda_i=\mu_i$ for $i\in B$.
\end{proof}

Lemma~\ref{l_drift} explains why~$Y$ can be thought of as a ``naive
martingale plus a corrector''. If there were no stable nodes, then $Y(x)=x$
and, for $k_0=1$, $Q$ would be an $\mathrm{M/M/1}$ queue whose arrival rate
equals its service rate, so that $Q(t)$ is a martingale up to the time it
hits~$0$. In the presence of stable nodes, the coordinates~$x_B$ are no
longer martingales: using $\lambda_i=\mu_i=a_i+\sum_n\lambda_np_{ni}$ for
$i\in B$, we obtain
\[
\LL x_i=\sum_{n\in S}p_{ni}\big(\mu_n\1{x_n\geq 1}-\lambda_n\big)
\qquad\text{whenever }x_j\geq 1\text{ for all }j\in B ,
\]
a drift which is of order~$1$, takes both signs, and vanishes only on
average, with respect to the product of the geometric laws
$\prod_{n\in S}(1-\rho_n)\rho_n^{x_n}$. By Lemma~\ref{l_drift}, the
corrector $u(x_S)$ solves the corresponding Poisson equation
$\LL u=-\LL x_B$ on $\{x_j\geq 1\ \forall j\in B\}$, which is the classical
device of homogenisation; what is special about Jackson networks is that
the corrector is explicit and \emph{linear}. That it is unbounded does no
harm, since its increments are bounded.

\medskip

Let us now fix $i\in B$ and put
\[
\tau_i:=\inf\{t\geq 0:Q_i(t)=0\} .
\]
Taking the $i$-th coordinate in Lemma~\ref{l_drift}, we obtain
\begin{equation}
\label{eq_coorddrift}
\LL Y_i(x)=\1{x_i=0}\,\mu_i(1-\tilde p_{ii})
 -\sum_{j\in B\setminus\{i\}}\1{x_j=0}\,\mu_j\tilde p_{ji} ;
\end{equation}
in particular, $\LL Y_i\leq 0$ on $\{x_i\geq 1\}$, with equality if
$k_0=1$. Thus, up to time~$\tau_i$, the process $Y_i(Q)$ has nonpositive
drift and bounded jumps, and we shall see that its jump variance is bounded
away from zero. The following tail bound is then a consequence
of~\cite[Theorems~2.4.5 and~2.4.8]{MPW17}, applied to the jump chain
of~$Q$.

\begin{prop}
\label{p_tail}
There is $C<\8$ such that
\[
\IP_x[\tau_i>t]\leq C\,Y_i(x)\,t^{-1/2}
\qquad\text{for all }i\in B,\ x\in\Z_+^d\text{ and }t\geq 1 .
\]
In particular, $\IE_x\tau_i^\gamma<\8$ for every $\gamma<\frac12$, and
$\liminf_{t\to\8}Q_i(t)=0$, $\IP_x$-a.s.
\end{prop}

\begin{proof}
Put $a_*:=\sum_\ell a_\ell$, so that $0<a_*\leq q(y)\leq\Lambda$ for
every~$y$, with $\Lambda=\sum_ja_j+\sum_j\mu_j$ as in
Section~\ref{s_sym}. Let $(X_n)_{n\geq 0}$ be the jump chain of~$Q$,
$\mathcal F_n:=\sigma(X_0,\dots,X_n)$,
$\Delta_n:=Y_i(X_{n+1})-Y_i(X_n)$, and
$\hat\tau:=\min\{n\geq 0:(X_n)_i=0\}$. We may assume that $x_i\geq 1$, since
otherwise $\tau_i=0$; then $Y_i(x)\geq 1$. We have $|\Delta_n|\leq 1$, and,
on $\{n<\hat\tau\}$, $\IE[\Delta_n\mid\mathcal F_n]=\LL Y_i(X_n)/q(X_n)\leq 0$
by~\eqref{eq_coorddrift}, whereas, counting only the service completions at
node~$i$ and using~\eqref{eq_harm} and Jensen's inequality,
\[
\IE[\Delta_n^2\mid\mathcal F_n]
\geq\frac{\mu_i}{\Lambda}\sum_{\ell\in\Vs}p_{i\ell}(1-\Gamma_{\ell i})^2
\geq\frac{\mu_i(1-\tilde p_{ii})^2}{\Lambda}
\geq v:=\frac1\Lambda\min_{j\in B}\mu_j(1-\tilde p_{jj})^2>0 .
\]
Fix $s\geq 1$, let $\sigma_s:=\min\{n\geq 0:Y_i(X_n)\geq s\}$ and
$\zeta:=\hat\tau\wedge\sigma_s$. Since $Y_i(X_{n\wedge\hat\tau})$ is a
nonnegative supermartingale, \cite[Theorem~2.4.5]{MPW17} gives
$\IP_x[\sigma_s<\hat\tau]\leq Y_i(x)/s$. Next, $f(y):=2(s+1)y-y^2$ satisfies
$f(y+\delta)-f(y)=2(s+1-y)\delta-\delta^2$, so that
$\IE\big[f(Y_i(X_{n+1}))-f(Y_i(X_n))\mid\mathcal F_n\big]\leq -v$ on
$\{n<\zeta\}$. Since $0\leq Y_i(X_{n\wedge\zeta})\leq s+1$, and~$f$ is
nonnegative on $[0,s+1]$, \cite[Theorem~2.4.8]{MPW17} gives
$v\,\IE_x(n\wedge\zeta)\leq f(Y_i(x))\leq 2(s+1)Y_i(x)$, so that
$v\,\IE_x\zeta\leq 2(s+1)Y_i(x)$. As
$\{\hat\tau>n\}\subset\{\sigma_s<\hat\tau\}\cup\{\zeta>n\}$, Markov's
inequality and the choice $s=n^{1/2}$ give
\[
\IP_x[\hat\tau>n]\leq\big(1+4v^{-1}\big)\,Y_i(x)\,n^{-1/2},\qquad n\geq 1 .
\]
Since the holding rates of~$Q$ are bounded away from zero (by~$a_*$) and
$Y_i(x)\geq 1$, the same bound for~$\tau_i$, with a larger~$C$, follows by an
elementary large-deviation estimate for sums of exponential random
variables. The moment bound follows by integrating the tail bound, and the
last claim by the Markov property at time $n\in\N$.
\end{proof}

For $k_0=1$ and $S=\emptyset$, $\tau_1$ is the busy period of a critical
$\mathrm{M/M/1}$ queue, so the exponent~$\frac12$ cannot be improved in
general. Note also that, for $k_0=1$, the face $\{x_1=0\}$ is the whole
set where the balanced queue is empty, and~$m$ has finite mass on it;
combined with Proposition~\ref{p_tail}, this gives an alternative proof of
the recurrence of~$Q$ for $k_0=1$, which we do not pursue.

For $k_0\leq 2$, the last assertion of Proposition~\ref{p_tail} follows
from Theorem~\ref{t_main}, since then~$Q$ is recurrent. For $k_0\geq 3$,
however, it is new, and together with Theorem~\ref{t_main} it gives the
following picture.

\begin{cor}
\label{c_turns}
Let $k_0\geq 3$. Then, $\IP_x$-a.s.\ for every $x\in\Z_+^d$,
\[
\lim_{t\to\8}|Q_B(t)|=\8,\qquad\text{whereas}\quad
\liminf_{t\to\8}Q_i(t)=0\quad\text{for every }i\in B .
\]
\end{cor}

That is, the total number of customers at the balanced nodes tends to
infinity, but each of the balanced queues is empty 
at arbitrarily large times.

\begin{rem}
\label{r_k2}
The simultaneous emptying of all balanced queues corresponds to the corner
$\{x_B=0\}$, which, by Theorem~\ref{t_main}, is visited at arbitrarily large
times if $k_0\leq 2$, and only finitely many times if $k_0\geq 3$. For
$k_0=2$, a direct proof of recurrence by the Lyapunov-function method would
require a function with finite sublevel sets whose drift is nonpositive
outside a finite set, cf.~\cite[Theorem~2.5.2]{MPW17}. It seems that such a
function can indeed be constructed, but the construction we have is lengthy
and adds little to Theorem~\ref{t_main}; let us only indicate its
structure. Let~$R$ be the $B\times B$ matrix with rows~$r_i$, $i\in B$, and
let
$\mathcal S(x):=\sum_yq(x,y)\big(Y(y)-Y(x)\big)^\top\big(Y(y)-Y(x)\big)$
be the covariance matrix of~$Y$. A computation similar to the proof of
Lemma~\ref{l_drift} shows that, on $\{x_j\geq 1\ \forall j\in B\}$,
$\mathcal S(x)$ equals $\bar\Sigma:=R+R^\top$ plus a term which is linear in
$\big(\mu_n\1{x_n\geq 1}-\lambda_n\big)_{n\in S}$ (and hence vanishes on
average), and that~$\bar\Sigma$ is positive definite; this holds for
every~$k_0$. The identity $\bar\Sigma=R+R^\top$ is the skew symmetry
condition of Harrison and Williams~\cite{HW87}, under which the Lebesgue
measure is invariant for the reflected Brownian motion in the orthant with
covariance matrix~$\bar\Sigma$ and reflection directions~$r_i$ (compare
with the fact that~$m$ does not depend on~$x_B$). For $k_0=2$, it places~$Q$
exactly on the recurrence/transience borderline for zero-drift reflected
random walks in the quarter plane~\cite{AFM95,FMM92} (the
parameter~$\alpha$ of~\cite{Wil85} vanishes), so that the Lyapunov function
has to be of ``sublogarithmic'' type. Its leading part is
$\log\big(C_1+h(Y(x))\big)$, where~$h$ is the logarithm of the
$\bar\Sigma^{-1}$-norm plus a suitable multiple of the corresponding polar
angle, so that the derivatives of~$h$ along~$r_i$ vanish on the faces
$\{y_i=0\}$; using~$Y(x)$ rather than~$x_B$ as the argument accounts for the
corrector~$u$. The concavity of $\log(C_1+\cdot)$ produces, in the
interior, a negative drift of order only $|Y|^{-2}(\log|Y|)^{-2}$, whereas the
fluctuations of~$\mathcal S(x)$ around~$\bar\Sigma$ produce a term of the
larger order $|Y|^{-2}(\log|Y|)^{-1}$, which has no definite sign. This term cannot be
dominated, and has to be cancelled by a \emph{second} corrector, linear
in~$x_S$ and solving a Poisson equation similar to the one for~$u$. Finally,
a small angular perturbation of~$h$ makes the derivatives on the faces
strictly negative, and an additional term controls the states in
which~$|x_S|$ is large.
\end{rem}

\begin{rem}
\label{r_k3}
For $k_0\geq 3$, by the transience criterion~\cite[Theorem~2.5.8]{MPW17}
(applied to the jump chain of~$Q$), $Q$ is transient if and only if there
are a function $f:\Z_+^d\to\R_+$ and a nonempty set $A\subset\Z_+^d$, not
necessarily finite, such that $\LL f\leq 0$ on $\Z_+^d\setminus A$ and
$f(y)<\inf_Af$ for some $y\notin A$. Once transience is known, such a
function always exists, e.g.\ $x\mapsto\IP_x[Q\text{ hits }x_0]$ with
$A=\{x_0\}$; the question is whether an \emph{explicit} one can be found.
With $\bar\Sigma=R+R^\top$ as in Remark~\ref{r_k2} and
$\varrho(y):=(y\bar\Sigma^{-1}y^\top)^{1/2}$, the natural candidate is
$\varrho(Y)^{-\alpha}$, with correctors as in Remark~\ref{r_k2}. In the
interior,
$\sum_{i,j\in B}\bar\Sigma_{ij}\,\partial_i\partial_j\varrho^{-\alpha}
=-\alpha(k_0-2-\alpha)\varrho^{-\alpha-2}$, which is negative for
$0<\alpha<k_0-2$: this is the gain of dimension. On the face
$\{y_i=0\}$, however, the derivative of~$\varrho^{-\alpha}$ along~$r_i$
equals $-\alpha\varrho^{-\alpha-2}\big(R_a\bar\Sigma^{-1}y^\top\big)_i$, where
$R_a:=\frac12(R-R^\top)$ is the antisymmetric part of~$R$, and this has no
sign in general. In some cases the difficulty is easily circumvented,
e.g.\ if $\tilde P=0$ (no customer ever travels from one balanced node to
another, directly or through~$S$), when $R_a=0$ and a small modification
of~$\varrho^{-\alpha}$ works. In general, one may look for a Lyapunov
function of the form $\varrho^{-\alpha}\psi$, with~$\psi$ a function of the
angular variable only; this leads to an oblique-derivative subsolution
problem on a spherical simplex, whose corners are the main difficulty. This
problem can be solved in specific examples, but we do not know how to
solve it in general for all networks with $k_0\geq 3$.
\end{rem}

\section*{Acknowledgements}
This paper was written together with Claude Opus 5.5, an AI model developed by
Anthropic, in the course of many conversations: the arguments were
discussed, developed and checked, and large parts of the text were drafted
jointly with Claude. The author has verified all the proofs and takes full
responsibility for the content of the paper.

The author was partially supported by
CMUP, member of LASI, which is financed by national funds
through FCT --- Funda\c{c}\~ao
para a Ci\^encia e a Tecnologia, I.P., 
under the project with reference UID/00144/2025,
\url{https://doi.org/10.54499/UID/00144/2025}.

\end{document}